\documentclass[11pt]{amsart}

\usepackage[T1]{fontenc}
\usepackage{lmodern}
\usepackage{geometry}
\usepackage{amsmath}
\usepackage{amssymb}
\usepackage{amsfonts}
\usepackage{mathrsfs}
\usepackage{enumitem}
\usepackage{graphicx}

\theoremstyle{plain}
\newtheorem{theorem}{Theorem}[section]
\newtheorem{lemma}[theorem]{Lemma}
\newtheorem{corollary}[theorem]{Corollary}

\theoremstyle{definition}
\newtheorem{definition}[theorem]{Definition}

\theoremstyle{remark}

\title{About Smooth and Non-Poor Subspaces of Daugavet Spaces}

\author{Samir Hamad}

\email{samih49@zedat.fu-berlin.de}

\date{April 2026}


\begin{document}

\begin{abstract}
We discuss an example of a non-complete normed space with the Daugavet property such that the norm is Gâteaux differentiable at every nonzero point. In contrast, we note that the dual norm of a normed space with the Daugavet property is not Gâteaux differentiable at any point. Furthermore, we show that quasilacunary Müntz spaces form a natural class of subspaces of $C[0,1]$, isomorphic to $c_0$, for which the corresponding quotient spaces fail to have the Daugavet property. At the same time, the slice diameter two property is preserved under this construction.
\end{abstract}

\maketitle

\section{Introduction}
We say that a normed space $X$ has the \emph{Daugavet property}, shortly $X \in \mathrm{DPr}$, if
\[
\|\mathrm{Id} + T\| = 1 + \|T\|,
\]
for every rank-one operator $T \colon X \to X$. By passing to adjoint operators, the following characterisation was obtained in \cite[Lemma~2.2]{kadets2000banach}.

\begin{theorem}
Let $X$ be a normed space. The following assertions are equivalent:
\begin{enumerate}
\item $X$ has the Daugavet property.
\item For every $y \in S_X$, $x^* \in S_{X^*}$, and every $\varepsilon > 0$, there exists $y^* \in S_{X^*}$ such that
\[
y^*(y) \ge 1 - \varepsilon
\quad \text{and} \quad
\|x^* + y^*\| \ge 2 - \varepsilon.
\]
\item For every $\varepsilon > 0$ and every $y \in S_X$, the closed convex hull of
\[
\{\, u \in (1+\varepsilon) B_X \mid \|y + u\| \ge 2 - \varepsilon \,\}
\]
contains $S_X$.
\end{enumerate}
\end{theorem}

This characterisation allows to observe that every slice of $B_X$, i.e., sets of the form
\[
\operatorname{Slice}(B_X,x^*,\varepsilon) := \{ x \in B_X \mid x^*(x) > 1 - \varepsilon\},\]
for $x^* \in S_{X^*}$ and $\varepsilon > 0$, has diameter $2$, i.e., $X$ has the \emph{slice diameter two property}. A natural question is for which subspaces of a Daugavet space the corresponding quotient still has the Daugavet property. A central tool for this investigation is Theorem~1.3, which is based on the in~\cite{narrowrich} introduced theory of operators between Banach spaces $T : X \rightarrow Y$, that are narrow with respect to a pair $(G, \Gamma) \subset L(X,Y) \times \mathcal{P}(S_{X^*})$, i.e., for every $x \in S_X$, $y \in S_Y$, every $x^* \in \Gamma$, and every $\varepsilon > 0$ there exists $z \in S_X$ such that
\[
\|Gz + y\| > 2 - \varepsilon \quad \text{and} \quad \|T(x - z)\| + |x^*(x - z)| < \varepsilon.
\]
In the case of $G = \operatorname{Id}$ we say that $T$ is narrow with respect to $\Gamma$, and if also $\Gamma = S_{X^*}$ then we say that $T$ is narrow. For a Banach space $X$ with the Daugavet property we say that a closed subspace $Y \subset X$ is rich with respect to $\Gamma \subset S_{X^*}$ if the quotient map $q : X \rightarrow X/Y$ is narrow with respect to $\Gamma$, and rich if $\Gamma = S_{X^*}$. It was shown in~\cite[Theorem~5.12]{narrowrich} that a closed subspace $Y \subset X$ is rich if $Y \subset \tilde{Y} \subset X$ implies that $\tilde{Y}$ has the Daugavet property.

The following notion, which dualises the concept of rich subspaces, was introduced in \cite{quotient}.

\begin{definition}
Let $X \in \mathrm{DPr}$.  
A closed subspace $Z \subset X$ is said to be \emph{poor} if
\[
X / \tilde{Z} \in \mathrm{DPr}
\quad \text{for every closed subspace } \tilde{Z} \subset Z.
\]
\end{definition}

It turns out that a closed subspace of $Z \subset X$ is poor if the quotient map from $X^*$ to $X^*/Z^\perp$ is narrow with respect to $S_X$.

\begin{theorem}\cite[Theorem~5.8]{quotient} Let $X$ be a Banach space enjoying the Daugavet property. For a closed subspace $Z \subset X$, the following conditions are equivalent:
\begin{enumerate}
    \item $Z$ is poor.
    \item $X / \tilde{Z} \in \mathrm{DPr}$ for every closed subspace $\tilde{Z} \subset Z$ of codimension $\mathrm{codim}_Z(\tilde{Z}) \le 2$.
    \item $Z^{\perp}$ is a subspace of $X^*$ that is rich with respect to $S_X$.
    \item For every $x^*, e^* \in S_{X^*}$, $\varepsilon > 0$, and for every $x \in S_X$ such that $e^*(x) > 1 - \varepsilon$,  
    there exists an element $v^* \in B_{X^*}$ satisfying
    \[
    v^*(x) > 1 - \varepsilon, \quad 
    \|x^* + v^*\| > 2 - \varepsilon, \quad 
    \|(e^* - v^*)|_Z\| < \varepsilon,
    \]
    that is, the quotient map from $X^*$ onto $X^* / Z^{\perp}$ is narrow with respect to $S_X$.
\end{enumerate}
\end{theorem}

By comparison with Theorem~1.1, we observe that condition~(4) can be rephrased as a strengthened form of the Daugavet property, which makes the study of poor subspaces particularly natural. Theorem~1.3 was used in \cite{quotient} to show that the reflexive subspaces of a Banach space with the Daugavet property are poor, a result originally due to Shvydkoy \cite{shvydkoy2000geometric}. Another important result in the study of poor subspaces is that $C[0,1]$ and $L^1[0,1]$ both contain a subspace isomorphic to $\ell^1$ that is not poor \cite[Theorem~6.10]{quotient}, showing that the Radon–Nikodým property does not imply poverty. In the same paper, it was left open whether closed subspaces with separable duals, or those not containing isomorphic copies of $\ell^1$, are necessarily poor. We will provide a simple subspace of $C[0,1]$ isomorphic to $c_0$ that is not poor in $C[0,1]$, thereby giving negative answers to both of the above questions.

Independent of this, we first address the question of whether a normed space with the Daugavet property necessarily has non-smooth points, as Theorem~1.1 might suggest. We recall that for a normed space $X$ and $x \in X \setminus \{0\}$, the norm of $X$ is called \emph{G\^ateaux differentiable at $x$} 
(or $x$ is a smooth point) if
\[
\lim_{t \to 0}
\frac{\|x + t h\| + \|x - t h\| - 2\|x\|}{t}
= 0
\quad \text{for every } h \in S_X.
\]

A fundamental tool for detecting smooth points is given by Shmulyan's test, which characterises smoothness in terms of unique supporting functionals. This result can be found in \cite[Corollary~7.22]{fabian2011banach}.

\begin{theorem}[Shmulyan's test]
Let $X$ be a normed space and let $x \in X \setminus \{0\}$. Then $x$ is a smooth point if and only if there exists a unique functional $f \in S_{X^*}$ such that
$f(x) = \|x\|$.
\end{theorem}

Using the following result, which asserts that in the dual of a Daugavet space, given any point on the unit sphere and any separable subspace of the dual, one can find a functional that nearly attains its norm at the chosen point while simultaneously being codirected with every element of the separable subspace, we can show that the dual of a Daugavet space does not contain any smooth points.

\begin{lemma}\cite[Lemma~5.1]{quotient}
Let $X$ be a normed space with the Daugavet property. Then for every $x \in S_X$, every $\varepsilon > 0$, and every separable subspace $V \subset X^*$, there exists $x^* \in S_{X^*}$ such that  
\[
\operatorname{Re} x^*(x) \ge 1 - \varepsilon
\quad \text{and} \quad
\|x^* + f\| = 1 + \|f\| \quad \text{for all } f \in V.
\]
\end{lemma}

\section{Smoothness in Daugavet spaces}

The original proof of the following theorem was simplified by A. Rueda Zoca.

\begin{theorem}
If $X$ is a normed space with $X \in DPr$, then $X^*$ does not have any smooth points.
\end{theorem}

\begin{proof}
Assume that there exists a smooth point $v^{*} \in S_{X^{*}}$. Let $V = \operatorname{span}\{v^{*}\}$ and choose $x^{*} \in S_{X^{*}}$ according to the previous lemma. Then
\[
\|v^{*}+x^{*}\| = 2
\quad\text{and}\quad
\|v^{*}-x^{*}\| = 2.
\]
Hence there exist $F,G \in S_{X^{**}}$ such that
\[
F(v^{*}+x^{*})=2
\quad\text{and}\quad
G(v^{*}-x^{*})=2.
\]
It follows that
\[
F(v^{*})=G(v^{*})=1,
\quad
F(x^{*})=1,
\quad
\text{and}
\quad
G(x^{*})=-1.
\]
Thus $F \neq G$, and both functionals norm $v^{*}$. This contradicts Shmulyan's test. Therefore, $X^*$ has no smooth points.
\end{proof}

The next theorem does not follow from the previous one, since V.~Kadets and D.~Werner constructed in \cite[Theorem 3.3]{schur} an example of a Daugavet space whose ultrapower does not have the Daugavet property by taking the $\ell_1$-sum of modified Bourgain-Rosenthal spaces. 

A.~Rueda Zoca also showed in \cite{slice} that for every $\varepsilon>0$ there exists a Banach space $X$ with the Daugavet property such that, for every free ultrafilter $\mathcal{U}$ over $\mathbb{N}$, the ultrapower $X^{\mathcal{U}}$ has a slice of diameter less than $\varepsilon$. For background on ultraproducts, we refer to \cite{zbMATH03640303}.

\begin{theorem}
If $X$ is a Banach space with $X \in DPr$ and $\mathcal{U}$ is a free ultrafilter on $\mathbb{N}$, then $X^{\mathcal{U}}$ has no smooth points.
\end{theorem}

\begin{proof}
Suppose $x = [x_1, x_2, \dots] \in S_{X^{\mathcal{U}}}$ is a smooth point. We can assume that $\mathcal{k}x_i \mathcal{k} = 1$ for every $i \in \mathbb{N}$. Then, for $f_i \in S_{X^*}$ with $f_i(x_i) > 1 - 1/i$, we have for $f = [f_1, f_2, \dots] \in S_{(X^{\mathcal{U}})^*}$ that $f(x) = 1$. By Lemma~1.5, we can choose $g_i \in S_{X^*}$ with $g_i(x_i) > 1 - 1/i$ and $\mathcal{k}g_i - f_i\mathcal{k} = 2$ for every $i$. Now, for $g = [g_1, g_2, \dots] \in S_{(X^{\mathcal{U}})^*}$ we have $g(x) = 1$, while obviously $g \neq f$. This contradicts Shmulyan's test.
\end{proof}

Let $X \subset L^{1}[0,1]$ be a subspace and let $f \in S_{X}$ satisfy
\[
\mathcal{L}\big(\{x \in [0,1] \mid f(x)=0\}\big)=0,
\]
where $\mathcal{L}$ denotes the one-dimensional Lebesgue measure. Then there exists a unique $g \in S_{L^{\infty}[0,1]}$ such that
\[
\int_{0}^{1} f(x) g(x)\,dx = 1.
\]
By Shmulyan's test, $f$ is a smooth point of $L^{1}[0,1]$ and hence also of $X$. With this criterion, we can construct a non-complete subspace of $L^{1}[0,1]$ having the Daugavet property that consists entirely of smooth points.
\begin{theorem}
There is a smooth non-closed subspace $X \subset L^1[0,1]$ with $X \in DPr$.
\end{theorem}

\begin{proof}
Let $(\varepsilon_i)_{i \in \mathbb{N}}$ be a sequence with $\varepsilon_i > 0$ for all $i$ and $\varepsilon_i \to 0$. We will construct a sequence of finite-dimensional subspaces $(X_i)_{i \in \mathbb{N}}$ of $L^1[0,1]$ satisfying $X_i \subset X_{i+1}$ for all $i$, and such that for each $i$ there exists a finite cover
\[
\bigcup_{m=1}^{m_i} I_m^i = [0,1],
\]
where each $I_m^i$ is a closed interval, any two distinct intervals intersect only at their endpoints, and every $f \in X_i$ is a polynomial when restricted to each interval $I_m^i$. Moreover, we require that every $f \in S_{X_i}$ is nonzero on every interval $I_m^i$. We construct the spaces in such a way that for every $g,f \in S_{X_i}$ there exists a finite family $\{h_j\}_{j=1}^{n_i} \subset S_{X_{i+1}}$ such that
\[
\left\lVert \sum_{j=1}^{n_i} \lambda_j h_j - f \right\rVert < \varepsilon_i
\quad\text{and}\quad
\|h_j + g\| > 2 - \varepsilon_i
\quad \text{for every } j=1,\dots,n_i,
\tag{1}
\]
for some scalars $\lambda_1,\dots,\lambda_{n_i}$ with
\[
\sum_{j=1}^{n_i} \lambda_j = 1.
\]

We start with $X_1 = \operatorname{span}\{1\}$, the span of the constant function $1$. Suppose that the subspaces $X_1, \dots, X_n$ have already been constructed. Let $v_1, \dots, v_{j_n} \in S_{X_n}$ be an $\varepsilon_n/2$-net in $S_{X_n}$, and let $\{I_m^n\}_{m=1}^{m_n}$ be a partition of $[0,1]$ such that every $f \in S_{X_n}$ is a nonzero polynomial on each interval $I_m^n$. Since $X_n$ is finite-dimensional, we can choose $q \in \mathbb{N}$ large enough such that
\[
\deg\!\bigl(f|_{I_m^n}\bigr) < q
\quad \text{for every } m=1,\dots,m_n \text{ and every } f \in S_{X_n}.
\tag{2}
\]

Now choose $\delta>0$ so small that
\[
\int_B |v_i|\,dx < \varepsilon_n/4,
\]
for every measurable set $B \subset [0,1]$ with $\mathcal{L}(B)<\delta$
and every $i=1,\dots,j_n$.

We take a cover of $[0,1]$,
\[
\bigcup_{j=1}^{m_{n+1}} I^{n+1}_j = [0,1],
\]
such that each closed interval $I^{n+1}_j$ is contained in some $I^n_m$, 
$\mathcal{L}(I^{n+1}_j) < \delta$ for every $j$, 
and the intervals $I^{n+1}_j$ intersect only at their endpoints.

Denote
\[
\tilde{w}_{i,j}
:=
\frac{v_i \chi_{I^{n+1}_j}}
{\left\| v_i \chi_{I^{n+1}_j} \right\|}.
\]
Then we have
\begin{align*}
\| \tilde{w}_{i,j} + v_r \| 
&= \int_{I_j^{n+1}} \lvert\tilde{w}_{i,j} + v_r\rvert \, dx + \int_{[0,1] \setminus I_j^{n+1}} \lvert v_r \rvert \, dx \\
&> \int_{I_j^{n+1}} \lvert\tilde{w}_{i,j}\rvert \, dx - \int_{I_j^{n+1}} \lvert v_r \rvert \, dx + 1 - \varepsilon_n/4 \\
&> \int_{I_j^{n+1}} \lvert\tilde{w}_{i,j}\rvert \, dx + 1 - \varepsilon_n/2 \\
&> 2 - \varepsilon_n/2,
\tag{3}
\end{align*}
for every $ r,i \in \{1,\dots,j_n\} \text{ and } j \in \{1,\dots,m_{n+1}\}$.

Moreover, the following decomposition holds:
\[
\sum_{j=1}^{m_{n+1}}
\| v_i \chi_{I^{n+1}_j} \|
\, \tilde{w}_{i,j}
= v_i
\quad \text{for every } i \in \{1,\dots,j_n\}.
\tag{4}
\]

Define
\[
w_{i,j}
:=
\frac{\tilde{w}_{i,j} + p\, x^{q+(i-1)m_{n+1}+j}}
{\left\|
\tilde{w}_{i,j} + p\, x^{q+(i-1)m_{n+1}+j}
\right\|},
\]
where $q$ is chosen as in (2) and $p > 0$ is sufficiently small such that (3) subsists for the new family $w_{i,j}$:
\[
\| w_{i,j} + v_r \| > 2 - \varepsilon_n/2
\quad \text{for every } r,i \in \{1,\dots,j_n\}
\text{ and } j \in \{1,\dots,m_{n+1}\},
\]
and (4) gets into
\[
\left\|
\sum_{j=1}^{m_{n+1}}
\| v_i \chi_{I^{n+1}_j} \|
\, w_{i,j}
- v_i
\right\|
<
\frac{\varepsilon_n}{2},
\]
for every $i$. Because of the way we chose $q$, we note that $w_{i,j}$ is a nonzero polynomial on every interval $I^{n+1}_j$. Now define
\[
X_{n+1}
:=
\operatorname{span}\Bigl(
X_n,
\{\, w_{i,j} \mid i=1,\dots,j_n,\; j=1,\dots,m_{n+1} \}
\Bigr).
\]
Obviously, $X_n \subset X_{n+1}$. With the partition
$\bigcup_{j=1}^{m_{n+1}} I^{n+1}_m = [0,1]$, every
$f \in X_{n+1}$ is a polynomial on each $I^{n+1}_j$, since each
$I^{n+1}_j$ is contained in some $I^n_{m}$. Property (1) follows from the last two inequalities and the fact that
$(v_i)$ is an $\varepsilon_n/2$-net of $S_{X_n}$. 

It remains to show that every $f \in S_{X_{n+1}}$ is nonzero on each $I^{n+1}_j$.
Suppose, by contradiction, that there exists
\[
f
=
x_n
+
\sum_{i,j} \lambda_{i,j} w_{i,j}
\in S_{X_{n+1}}
\]
which vanishes on, without loss of generality, $I^{n+1}_1$, where
$x_n \in S_{X_n}$.

By construction, each $w_{i,j}$ has higher degree than every element in $X_n$, and
$w_{u,v}$ has higher degree than $w_{s,t}$ whenever
$u>s$ or $u=s$ and $v>t$. Since $f$ vanishes on $I^{n+1}_1$ and hence at
infinitely many points, all coefficients $\lambda_{i,j}$ must be zero.
Indeed, otherwise $f$ would contain a nontrivial highest-degree term and
could not vanish at infinitely many points.

Thus we obtain $f=x_n \in S_{X_n}$.
By the induction hypothesis, such a function cannot vanish identically on
$I^{n+1}_1$, since this interval is contained in some
$I^n_m$, and vanishing on $I^{n+1}_1$ would imply infinitely many zeros
on $I^n_m$, hence $f\equiv 0$ on $I^n_m$, a contradiction. 

Now define $X := \bigcup_{i=1}^\infty X_i$. By property (1), the space $X$ has the Daugavet property. 

Let $f \in S_X$. Then $f \in S_{X_n}$ for some $n$, and hence there exists a finite partition of $[0,1]$ such that $f$ is a polynomial on each subinterval. Since $f$ is a nonzero polynomial on each of these subintervals, its set of zeros has Lebesgue measure zero, i.e.,
\[
\mathcal{L}(\{x \in [0,1] \mid f(x)=0\})=0.
\]
Therefore, by the remark preceding the Theorem, $f$ is a smooth point.
\end{proof}

Note that the completion of this normed space is not smooth. Moreover, it is not even known whether every closed subspace of $L^1[0,1]$ with the Daugavet property contains a non-smooth point. Similarly, it remains open whether there exists a strictly convex Banach space with the Daugavet property. In fact, V.~Kadets constructed a non-complete normed space that is strictly convex and possesses the Daugavet property in \cite{zbMATH00911877}.

\section{Poor subspaces}
We now turn to the study of poor subspaces. Let $\{\lambda_i\}_{i=1}^{\infty}$ be an increasing sequence of positive real numbers. Moreover, let $M(\Lambda) = \{x^{\lambda_i}\}_{i = 1}^{\infty}$ be the sequence of the functions $x^{\lambda_i}$ on $[0,1]$ and let $\overline{\operatorname{span}}^EM(\Lambda)$ be the closed linear span of $M(\Lambda)$ in a given Banach space $E$ containing $M(\Lambda)$, which we call a Müntz space. A famous theorem of C.~Müntz \cite{muntz1} states that for $E = C[0,1]$ or $E = L^p$, $1 \leq p < \infty$, we have
\[
\overline{\operatorname{span}}^EM(\Lambda) \neq E \quad \text{if and only if} \quad \sum_{k=1}^{\infty} \frac{1}{\lambda_k} < \infty.
\]
The next theorem shows that a mild condition on $\Lambda$ implies the poverty of the corresponding Müntz space in $C[0,1]$. The dual of $C[0,1]$ will be identified with the Banach space $M[0,1]$ of Radon measures on $[0,1]$, where $\lvert \mu \rvert$ denotes the variation of $\mu \in M[0,1]$ and $\|\mu\| = \lvert \mu \rvert([0,1])$ denotes the total variation.

\begin{theorem}
Let $\Lambda = \{\lambda_i\}_{i=1}^\infty$ be a strictly increasing sequence of positive real numbers with $\lambda_i \rightarrow \infty$. Then the subspace $Z :=\overline{\operatorname{span}}\{1, x^{\lambda_1}, x^{\lambda_2}, \dots\}$ is non-poor in $C[0,1]$.
\end{theorem}

\begin{proof}
Suppose that $Z$ is poor in $C[0,1]$, i.e., it satisfies condition~(4) of Theorem~1.3. Then, for $x^* = -\delta_1$, $e^* = \delta_1$, and $\varepsilon > 0$, there exists $v^* \in B_{M[0,1]}$ such that
\[
\|v^* - \delta_1\| > 2 - \varepsilon, \quad 
\|(\delta_1 - v^*)|_Z\| < \varepsilon. \tag{1}
\]

Since $x^{\lambda_i} \to \chi_{\{1\}}$ pointwise on $[0,1]$, the dominated convergence theorem yields
\[
\left| \int_0^1 x^{\lambda_i} \, d\delta_1(x) - \int_0^{1} x^{\lambda_i}\, dv^*(x) \right|
\;\longrightarrow\;
\left|1 - v^*(\{1\})\right|.
\]
By the second inequality in~(1), we obtain $1 - v^*(\{1\}) \le \varepsilon$ and hence 
\[
|v^*|\bigl([0,1]\setminus\{1\}\bigr) 
= |v^*|\bigl([0,1]\bigr) - |v^*(\{1\})| 
\leq 1 - (1 - \varepsilon) = \varepsilon.
\]
From the first inequality in~(1), we deduce
\[
2\varepsilon \ge \left| v^*(\{1\}) - 1 \right| + |v^*|\bigl([0,1]\setminus\{1\}\bigr)
\ge \|v^* - \delta_1\| > 2 - \varepsilon.
\]
Choosing $\varepsilon > 0$ sufficiently small, we obtain a contradiction.
\end{proof}

\begin{definition}
Let $\Lambda = \{\lambda_i\}_{i=1}^\infty$ be a strictly increasing sequence of positive real numbers. We say that $\Lambda$ is quasilacunary if there exists a strictly increasing sequence of indices $(i_k)$ and a constant $q > 1$ such that
\[
\inf_k \frac{\lambda_{i_{k+1}}}{\lambda_{i_k}} \ge q \quad \text{and} \quad \sup_k (i_{k+1} - i_k) < \infty.
\]
\end{definition}

\begin{corollary}
There exists a non-poor subspace of $C[0,1]$ that is isomorphic to $c_0$.
\end{corollary}
\begin{proof}
Let $\Lambda = \{\lambda_i\}_{i=1}^\infty$ be quasilacunary. Then $\lambda_i \to \infty$ and hence $Z := \overline{\operatorname{span}}\{1, x^{\lambda_1}, x^{\lambda_2}, \dots\}$ is non-poor. By \cite[Corollary~9.3.2]{muntz}, $Z$ is isomorphic to $c_0$.
\end{proof}

\begin{corollary}
There exists a closed subspace $Z$ of $C[0,1]$ with separable dual such that $C[0,1]/Z$ does not enjoy the Daugavet property.
\end{corollary}
\begin{proof}
Suppose that $C[0,1]/Y$ has the Daugavet property for every closed subspace $Y \subset C[0,1]$ with separable dual. Then every closed subspace of $C[0,1]$ with separable dual is poor in $C[0,1]$, which contradicts the previous corollary.
\end{proof}

Note that this does not hold for $L^1[0,1]$, since every closed subspace $Z$ without isomorphic copies of $\ell^1$ is reflexive by a classical result of Kadets and Pe\l{}czy\'nski ~\cite[Theorem~8]{pel}, and therefore $L^1[0,1]/Z$ has the Daugavet property by Shvydkoy's result. In the next theorem, we show that quasilacunary Müntz spaces are non-poor in $C[0,1]$ only because one can choose the functional $e^*$ in Theorem~1.3 to have atoms.
\begin{theorem}
Let $Y \subset C[0,1]$ be a subspace isomorphic to $c_0$. Then, for every $x^*, e^* \in S_{M[0,1]}$ with $e^*$ having no atoms, every $\varepsilon > 0$, and every $f \in S_{C[0,1]}$ satisfying $e^*(f) > 1 - \varepsilon$, there exists $v^* \in B_{M[0,1]}$ such that
\[
v^*(f) > 1 - \varepsilon, \quad 
\|x^* + v^*\| > 2 - \varepsilon, \quad 
\|(e^* - v^*)|_Y\| < \varepsilon.
\]
\end{theorem}

\begin{proof}
Let $x^*, e^* \in S_{M[0,1]}$ with $e^*$ having no atoms, $\varepsilon > 0$ and $f \in S_{C[0,1]}$ such that $e^*(f) > 1- \varepsilon$. Let $\{f_n\}_{n \in \mathbb{N}} \subset Y$ be a sequence such that $Y= \overline{\operatorname{span}}\{f_n \mid n \in \mathbb{N}\}$ and that is $C$-equivalent to the canonical basis $\{e_n\}_{n \in \mathbb{N}}$ of $c_0$, i.e.,
\[
\frac{1}{C} \sup_{n} |a_n| \;\le\; \Biggl\|\sum_{n=1}^{N} a_n f_n \Biggr\|_\infty \;\le\; C \sup_{n} |a_n| \tag{1}
\]
for all $N \in \mathbb{N}$ and scalars $a_1, \dots, a_N \in \mathbb{R}$. In particular,
\[
\frac{1}{C} \le \mathcal{k}f_n\mathcal{k} \le C \quad \text{for every } n \in \mathbb{N}.
\]

Let $\delta > 0$ and, for each $m \in \mathbb{N}$, define
\[
C_m := \Bigl\{ z \in [0,1] \;\Big|\; \sum_{i = m}^\infty |f_i(z)| > \delta \Bigr\}.
\]
Clearly, the sets satisfy $C_{m+1} \subset C_m$. Suppose that there exists $y \in \bigcap_{m=1}^\infty C_m$. Then we would have
\[
\sum_{i=1}^\infty |f_i(y)| = \infty,
\]
which contradicts~(1). Hence,
\[
\bigcap_{m=1}^\infty C_m = \emptyset,
\]
and we may choose some $l \in \mathbb{N}$ such that
\[
\lvert e^* \lvert(C_l) < \delta \quad \text{and} \quad \sum_{i=l}^\infty |e^*(f_i)| < \delta. \tag{2}
\]
Choose $k \in \mathbb{N}$ sufficiently large so that for every $j \in \{0, \ldots, 2^k - 1\}$ and every $i \in \{1, \ldots, l-1\}$,
\[
|f_i(x) - f_i(x')| + |f(x) - f(x')| < \frac{\delta}{l} 
\quad \forall\, x, x' \in \Bigl[\frac{j}{2^k}, \frac{j+1}{2^k}\Bigr). \tag{3}
\]

Let $A$ denote the union of those intervals $[j/2^k, (j+1)/2^k)$ for which $[j/2^k, (j+1)/2^k) \setminus C_l$ is finite. Since $e^*$ has no atoms, it follows from~(2) that $\lvert e^*\lvert(A) < \delta$.

Let us enumerate the remaining intervals in the partition according to their left endpoints. More precisely, let $B_1$ denote the interval $[j/2^k, (j+1)/2^k)$ with the smallest $j$ such that $[j/2^k, (j+1)/2^k) \setminus C_l$ is infinite, $B_2$ the interval with the second smallest such $j$, and so on. In this way, we obtain a finite set of intervals $\{B_p\}^r_{p =1}$ with $r \leq 2^k-1$. For each $p$, choose a point $x_p \in B_p \setminus C_l$ such that 
\[
|x^*(\{x_p\})| < |e^*|(B_p) \, \delta.
\] 
Define
\[
v^* := \sum_{p=1}^{r} e^*(B_p) \, \delta_{x_p}.
\]

Obviously, $\mathcal{k}v^*\mathcal{k} = \sum_{p=1}^r |e^*(B_p)| \le 1$, and
\begin{align*}
|(v^* - e^*)(f)| &\le \lvert e^*\lvert(A) + \sum_{p=1}^r \int_{B_p} |f(x_p) - f(x)| \, d|e^*|(x) \\
&\le \delta + \delta \sum_{p=1}^r \lvert e^* \lvert(B_p) = 2\delta.
\end{align*}
We obtain $v^*(f) > e^*(f)- 2\delta$ and therefore $\mathcal{k}v^*\mathcal{k} > e^*(f)- 2\delta$.
Using this, we estimate
\begin{align*}
\|v^* + x^*\|
&= |x^*|([0,1] \setminus \{x_1, \dots, x_r\}) + |v^* + x^*|(\{x_1, \dots, x_r\}) \\
&> 1 - \delta + \sum_{p=1}^{r} \left(|e^*(B_p)| - \delta|e^*|(B_p)\right) \\
&> 1 - \delta + \left(e^*(f)- 3\delta\right). \tag{4}
\end{align*}
Since $\{x_1, \dots, x_r\} \cap C_l = \emptyset$, we have
\begin{align*}
\sum_{i=l}^{\infty} |v^*(f_i)| 
&\leq \sum_{i=l}^{\infty} \sum_{p=1}^{r} \lvert e^*(B_p) \rvert |f_i(x_p)| \\
&= \sum_{p=1}^{r} \lvert e^*(B_p) \rvert \sum_{i=l}^{\infty} |f_i(x_p)| \\
&\leq \delta \sum_{p=1}^{r} |e^*(B_p)| \\
&= \delta.
\end{align*}

Moreover, for the first $l-1$ functions, we estimate
\[
\begin{aligned}
\sum_{i=1}^{l-1} |(v^* - e^*)(f_i)| 
&\le C \, \lvert e^* \lvert(A) + \sum_{i=1}^{l-1} \sum_{p=1}^{r} \int_{B_p} |f_i(x_p) - f_i(x)| \, d|e^* |(x) \\
&\le C \,\delta + \sum_{i=1}^{l-1} \sum_{p=1}^{r} \frac{\delta}{l} \, |e^*|(B_p) \\
&< C \delta + \delta.
\end{aligned}
\]
Therefore, for an arbitrary $g = \sum_{i=1}^\infty a_i f_i \in S_Y$, we have
\[
|(v^* - e^*)(g)| \le C \Biggl( \sum_{i=1}^{l-1} |(v^* - e^*)(f_i)| + \sum_{i=l}^\infty (|v^*(f_i)| + |e^*(f_i)|) \Biggr) \le C^2\delta+3C\delta .
\]
We recall all inequalities we obtained:
\begin{align*}
\|v\| &\leq 1, \\
v^*(f) &> e^*(f) - 2\delta, \\
\|v^* + x^*\| &> 1 - \delta + \left(e^*(f) - 3\delta\right), \\
|(v^* - e^*)(g)| &\leq C^2\delta + 3C\delta \quad \text{for all } g \in S_Y.
\end{align*}
Since $e^*(f) > 1 - \varepsilon$, we can choose $\delta > 0$ sufficiently small such that $v^*$ satisfies the required properties.
\end{proof}

\begin{theorem}
Let $Z$ be a closed subspace of $C[0,1]$ without isomorphic $\ell^1$-copies. Then $C[0,1]/Z$ has the slice diameter two property.
\end{theorem}

\begin{proof}
Suppose that there exist $p > 0$ and a slice $S := S(B_{C[0,1]/Z}, \mu, \varepsilon)$ with 
$\mu \in S_{(C[0,1]/Z)^*} = S_{Z^\perp}$ and $\varepsilon > 0$ such that
\[
\|k - g\|_{C[0,1]/Z} < 2 - p,
\tag{1}
\]
for every $k,g \in S$. Choose an open interval $I \subset [0,1]$ with $\lvert \mu \lvert(I) < \varepsilon/2$ and some $[h] \in B_{C[0,1]/Z}$ with 
\[
\int_{[0,1]} h \, d\mu > 1 - \varepsilon/2,
\]
such that $h \in C[0,1]$ is a 
representative of $[h]$ with $\mathcal{k}h\mathcal{k} \leq 1 $.

For a set of distinct points $x_1^+, \dots, x_n^+, x_1^-, \dots, x_n^- \in I$ let $g^+ \in B_{C[0,1]}$ with 
\[
g^+ = h \text{ on } [0,1] \setminus I, \quad 
g^+(x_i^+) = 1, \quad 
g^+(x_i^-) = -1 \text{ for } i=1,\dots,n.
\]
Let $g^- \in B_{C[0,1]}$ with 
\[
g^- = h \text{ on } [0,1] \setminus I, \quad 
g^-(x_i^+) = -1, \quad 
g^-(x_i^-) = 1 \text{ for } i=1,\dots,n.
\]
Obviously $[g^+], [g^-] \in S$. Therefore by (1) there is an $f \in 4B_{Z}$ such that
\[
\|g^+-g^- - f \| < 2- p.
\]
We immediately obtain
\[
f(x_i^+) > p \quad \text{and} \quad f(x_i^-) < -p \text{ for all } i=1,\dots,n. \tag{2}
\]
The remainder of the proof consists of an inductive construction of a sequence in $Z$ equivalent to the canonical basis of $\ell^1$, where in each step we apply the above construction to select suitable basis vectors.
In the initial step we pick some arbitrary $x_1 \in I$, so by~(2) we can pick $f_0 \in 4B_Z$ with $f_0(x_1) > p$. Suppose we have constructed \( f_0, f_1, \dots, f_m \in 4B_Z \) and points \( x_1, \dots, x_{2^m} \in I \) such that the following holds:
\[
|f_k(x_j)| > p \quad \text{ for all } j = 1, \dots, 2^m,  k = 0, 1, \dots, m,
\]
and for every choice of nonzero scalars \( a_0, \dots, a_m \), there exists an index \( i \in \{1, \dots, 2^m\} \) such that
\[
\operatorname{sign}(a_k\, f_k(x_i))
=
\operatorname{sign}(a_{k'}\, f_{k'}(x_i))
\quad \text{for all } k, k' \in \{0, 1, \dots, m\}. \tag{3}
\]

For each \( j \in \{1, \dots, 2^m\} \), choose an open set \( U_j \subset [0,1] \) containing \( x_j \) such that
\[
|f_k(x)| > p \quad \text{for all } x \in U_j \text{ and all } k = 1, \dots, m.
\]
Using the above construction which leads to (2), we can choose some \( f_{m+1} \in 4B_Z \) such that for every \( j \in \{1, \dots, 2^m\} \), there exist two points \( y^j_w, y^j_e \in U_j \) with
\[
f_{m+1}(y^j_w) > p \quad \text{and} \quad f_{m+1}(y^j_e) < -p.
\]
Let \( b_0, \dots, b_{m+1} \) be nonzero scalars. Then there exists some \( u \in \{0, \dots, 2^m\} \) such that
\[
\operatorname{sign}(b_k\, f_k(x_u))
=
\operatorname{sign}(b_{k'}\, f_{k'}(x_u))
\quad \text{ for all } k, k' \in \{1, \dots, m\}.
\]
Without loss of generality, let $b_1f_1(x_u) > 0$. Suppose $b_{m+1} > 0$. Using that \(f_k\) for \(k \in \{0, \dots, m\}\) does not switch its sign on \(U_u\), we obtain
\[
\operatorname{sign}(b_k\, f_k(y^u_w))= \operatorname{sign}(b_1\, f_1(y_w^u)) = 1 = \operatorname{sign}(b_{m+1}\, f_{m+1}(y_w^u)), 
\]
while for $b_{m+1} < 0$, we would use $y_e^u$ instead of $y_w^u$. We also have $\lvert f_k(y_w^u) \lvert > p$ for every $k \in \{1, \dots m+1\}$. By induction, we obtain a sequence \( (f_k)_{k \in \mathbb{N}} \) such that (3) holds for every finite subset $\{f_k\}_{\leq m}$.
Let $m \in \mathbb{N}$ be arbitrary again where we use the same notation as in (3). Then:
\[
\left\| \sum_{i=0}^{m} a_i f_i \right\|
\ge
 \left| \sum_{i=0}^{m} a_i f_i(x_i) \right| = \left| \sum_{i=0}^{m} \lvert a_if_i(x_i) \lvert \operatorname{sign}(a_{i}f_i(x_i))\right| 
>
p \sum_{i=0}^{m} |a_i|,
\]
The other inequality for $(f_i)$ being equivalent to the canonical $\ell^1$-basis follows from the fact that \( f_i \in 4B_Z \). Therefore we get a contradiction to the assumption of the theorem.
\end{proof}

For $L^\infty[0,1]$ the slice diameter two property is preserved under milder conditions.

\begin{theorem}
Let $Z \subset L^\infty[0,1]$ be a closed separable subspace. Then $L^{\infty}[0,1]/Z$ has the slice diameter two property.
\end{theorem}

\begin{proof}
We identify $L^\infty[0,1]^*$ with the space of bounded finitely additive signed measures on the Borel $\sigma$-algebra $\mathcal{B}([0,1])$ which are absolutely continuous with respect to the Lebesgue measure $\mathcal{L}$. 
Suppose that there exist $p > 0$ and a slice $S := S(B_{L^\infty[0,1]/Z}, \mu, \varepsilon)$ with $\mu \in S_{Z^\perp}$ and $\varepsilon > 0$ such that
\[
\|k - g\|_{L^\infty[0,1]/Z} < 2 - p,
\tag{1}
\]
for every $k,g \in S$.

Let $A \subset [0,1]$ be a Borel set with $\mathcal{L}(A) > 0$ such that the variation of $\mu$ on $A$ satisfies $|\mu|(A) < \varepsilon/2$. We may decompose $A$ into a disjoint union
\[
A = \bigsqcup_{i=1}^\infty A_i,
\]
where each $A_i \in \mathcal{B}([0,1])$ has strictly positive Lebesgue measure.

Let $[f] \in S$ be such that
\[
\int f \, d\mu > 1 - \varepsilon/2,
\]
where $f \in L^\infty[0,1]$ is a Borel measurable representative satisfying $\|f\|_\infty \leq 1$.

For an arbitrary partition $N_1 \cup N_2 = \mathbb{N}$, define
\[
h_{N_1,N_2}(x) =
\begin{cases}
f(x), & \text{if } x \notin A, \\
1, & \text{if } x \in A_j \text{ and } j \in N_1, \\
-1, & \text{if } x \in A_i \text{ and } i \in N_2,
\end{cases}
\]

\[
q_{N_1,N_2}(x) =
\begin{cases}
f(x), & \text{if } x \notin A, \\
-1, & \text{if } x \in A_j \text{ and } j \in N_1, \\
1, & \text{if } x \in A_i \text{ and } i \in N_2.
\end{cases}
\]
We have $h_{N_1,N_2}, q_{N_1,N_2} \in S$. By (1) there exists $w_{N_1,N_2} \in 4B_Z$ such that
\[
\|h_{N_1,N_2} - q_{N_1,N_2} - w_{N_1,N_2}\|_{L^\infty} < 2 - p,
\]
i.e., $w_{N_1,N_2} > p$ almost everywhere on every $A_j$ with $j \in N_1$, and $w_{N_1,N_2} < -p$ almost everywhere on every $A_i$ with $i \in N_2$. Assume that we have two different partitions of $\mathbb{N}$, say $(N_1,N_2)$ and $(\tilde{N}_1,\tilde{N}_2)$. Then there exists $m \in \mathbb{N}$ such that either
$m \in N_1$ and $m \in \tilde{N}_2$, which implies
\[
|w_{N_1,N_2} - w_{\tilde{N}_1,\tilde{N}_2}| > 2p
\quad \text{almost everywhere on } A_m,
\]
or $m \in N_2$ and $m \in \tilde{N}_1$, in which case we also obtain $|w_{N_1,N_2} - w_{\tilde{N}_1,\tilde{N}_2}| > 2p$ on $A_m$.

Since $p$ is fixed, $Z$ is separable, and the set of all partitions of $\mathbb{N}$ into two subsets is uncountable, we obtain a contradiction.

\end{proof}

\section*{Acknowledgements}
The author expresses his gratitude towards Dirk Werner, whose remarks greatly improved the presentation and quality of this paper, and likewise for the observation that the author's first example of a non-poor space isomorphic to a subspace of $c_0$ can be replaced by the more natural class of Müntz spaces, which under certain conditions are even isomorphic to $c_0$.

\bibliographystyle{plain}
\bibliography{mybib}

\end{document}